\documentclass{tran-l}

\numberwithin{equation}{section}

\newtheorem*{theom}{Theorem}
 \newtheorem{thm}{Theorem}[section]
 \newtheorem{lemma}[thm]{Lemma}
 \newtheorem{prop}[thm]{Proposition}
 \newtheorem{cor}[thm]{Corollary}

 \theoremstyle{definition}

\newtheorem*{example}{Example}
\newtheorem{exa}{Example}

\newtheorem{problem}{Problem}

\theoremstyle{remark}
\newtheorem{remark}[thm]{Remark}

\newenvironment{acknowledgment}[1][Acknowledgment]
 {\begin{trivlist} \item[\hskip \labelsep {\bfseries
#1}]}{\end{trivlist}}

\newcommand{ \Mf}{\mathfrak{M}}
\newcommand{\e}{\mathfrak{e}}
\newcommand{\Lc}{\mathcal{L}}

\newcommand{ \Bc}{\mathcal{B}}

\newcommand{\Sc}{\mathcal{S}}
\newcommand{\Gb}{\mathbb G}
\newcommand{ \Rb}{\mathbb{R}}
\newcommand{\Rf}{\mathfrak{R}}
\newcommand{\Sf}{\mathfrak{S}}

\newcommand{\Kc}{\mathcal{K}}
 
 \newcommand{ \Cbb}{\mathbb{C}}

 \newcommand{\al}{\alpha}
 \newcommand{\la}{\langle}
 \newcommand{\ra}{\rangle}

 \newcommand{\hot}{\hat{\otimes}}

\begin{document}

 \title[invariance subspace property and amenability]{Finite dimensional invariant subspace property and amenability for a class of Banach algebras}% (related to locally compact groups)}

\author[A. T.-M. Lau]{Anthony To-Ming Lau}% \dag}
\address{Department of Mathematical and Statistical sciences\\
           University of Alberta\\
           Edmonton, Alberta\\
           T6G 2G1 Canada}
\email{tlau@math.ualberta.ca}
\thanks{Supported by NSERC Grant MS100}

\author[Y. Zhang]{ Yong Zhang}% \ddag}
\address{Department of Mathematics\\
           University of Manitoba\\
           Winnipeg, Manitoba\\
           R3T 2N2 Canada}
\email{zhangy@cc.umanitoba.ca}
\thanks{Supported by NSERC Grant 238949}

\date{March 3, 2014}



\begin{abstract}
Motivated by a result of Ky Fan in 1965, we establish a characterization of a left amenable F-algebra (which includes the group algebra and the Fourier algebra of a locally compact group and quantum group algebras, or more generally the predual algebra of a Hopf von Neumann algebra) in terms of a finite dimensional invariant subspace property. This is done by first revealing a fixed point property for the semigroup of norm one positive linear functionals in the algebra. Our result answers an open question posted in Tokyo in 1993 by the first author (see \cite[Problem 5]{Lau_Tokyo}). We also show that the left amenability of an ideal in an F-algebra may determine the left amenability of the algebra.
\end{abstract}

\maketitle

\section{introduction}

In \cite{K3} (see also \cite{K1,K2,Lau_invar, L-P-W,Lau-Wong}), Ky Fan established the following remarkable ``Invariant Subspace Theorem" for left amenable semigroups:
\begin{theom}
Let $S$ be a left amenable semigroup, and let $\Sc = \{T_s: s\in S\}$ be a representation of $S$ as continuous linear operators on a separated locally convex space $E$. Then the following property holds:
\begin{description}
\item[(KF)]\label{K} If $X$ is a subset of $E$ (containing an $n$-dimensional subspace) such that $T_s(L)$ is an $n$-dimensional subspace contained in $X$ whenever $L$ is one and $s\in S$, and there exists a closed $\Sc$-invariant subspace $H$ in $E$ of codimension $n$ with the property that $(x+H)\cap X$ is compact convex for each $x\in E$, then there exists an $n$-dimensional subspace $L_0$ contained in $X$ such that $T_s(L_0) = L_0$ for all $s\in S$.
\end{description}
\end{theom}
The origin of Ky Fan's Theorem lies in the earlier investigations of Pontrjagin, Iovihdov, Krein and Naimark, concerning invariant subspaces for Lorentz transformations on a Hilbert space \cite{Iohv,I-K, Krein, Naimark1, Naimark2, Pontr}. In Physics, a Lorentz transformation is an invertible linear mapping on $\Rb^4$ that describes how a measurement of space and time observed in a frame of reference is converted into another frame of reference. From special relativity Lorentz transformations may be characterized as invertible linear mappings that preserve the quadratic form
\[ 
J(\vec x) = x^2 + y^2 + z^2 - c^2 t^2, \quad \vec x = (x, y, z, t)\in \Rb^4
\]
where the constant $c$ is the speed of light.
Quantity $J$ represents the space time interval. %Its invariance under Lorentz transformations is due to special relativity.
 It is a well known fact that for any Lorentz transformation $T$ there is a three dimensional subspace $V$ of $\Rb^4$ which is $T$-invariant and  positive (in the sense that $T(\vec x) = \vec x$ and $J(\vec x)\geq 0$ for all $\vec x \in V$). %There is also a one dimensional $T$-invariant subspace of $R^4$ which is negative.
 
L. S. Pontrjagin \cite{Pontr}, I. S. Iovihdov \cite{Iohv}, M. G. Krein \cite{Krein, I-K} and M. A. Naimark \cite{Naimark1,Naimark2} investigated infinite-dimension versions of the above invariant subspace property, and Naimark finally established the following theorem in 1963 \cite{Naimark1}.
\begin{theom}
Let $n>0$ be an integer. Consider the quadratic form on $\ell^2$ given by
\[
J_n(x) = \sum_{i=1}^n|x_i|^2 - \sum_{i=n+1}^\infty |x_i|^2, \quad x=(x_1,x_2,\cdots)\in \ell^2.
\]
Suppose that $G$ is a commutative group of continuous,  invertible, $J_n$-preserving, linear transformations on $\ell^2$. Then there is a $G$-invariant $n$-dimensional subspace $V$ of $\ell^2$ which is positive (in the sense that $J_n(x) \geq 0$ for all $x\in V$).
\end{theom}
 
To understand the conditions on the subset $X$ in Ky Fan's property (KF) we note that $X=\{x\in \ell^2: J_n(x)\geq0\}$ in the setting of the above theorem satisfies indeed the requirement in (KF). However, there is no mention of an invariant subspace $H$ in Naimark's result. Ky Fan's Theorem removes the commutative group and Hilbert space restrictions of Naimark's Theorem, replacing them by amenability and the conditions involving $H$.

%We refer to \cite{L-P-W} for the background of the theory.

The purpose of this paper is to establish a finite dimensional invariant subspace property similar to (KF) for the class of F-algebras. The class of F-algebras includes many classical Banach algebras related to a locally compact group or a hypergroup. It also includes the predual algebras of Hopf von Neumann algebras, in particular the class of quantum group algebras $L^1(\Gb)$. To achieve our goal we will first study the semigroup of all norm one normal positive functionals on the dual von Neumann algebra $A^*$ with the norm topology carried from $A$. We will show that the left amenability of an F-algebra $A$ is equivalent to 
the extreme left amenability of this topological semigroup. This equivalence plays an important role in our investigation about the Ky Fan's property for F-algebras.

The paper is organized as follows: In Section~\ref{sec 2} we discuss some basic properties of left amenability for an F-algebra that we shall need in establishing our main results. In Section~\ref{sec 3} we show (Theorem~\ref{ELA}) that an F-algebra $A$ is left amenable if and only if the semigroup $S$ of normal positive functionals of norm $1$ on $A^*$ has a fixed point property on a compact Hausdorff space. The latter property is equivalent to $S$ being extremely left amenable. In Section~4 we shall prove our main result (Theorem~\ref{KF thm}) about the finite dimensional subspace property of the F-algebra $A$ when $A$ is left amenable. Special cases regarding Banach algebras associated to a locally compact group or hypergroup and the quantum group algebra of a locally compact quantum group are also addressed there. In Section~\ref{sec 5} we investigate further the relation between the left amenability of an F-algebra and that of its closed left ideals. Applications to the group algebra $L^1(G)$, the measure algebra $M(G)$ and the Fourier Stieltjes algebra $B(G)$ are given. In Section~\ref{sec 6} we shall present some related results and discuss briefly the operator amenability for F-algebras. We will also address some open problems in this section.

Finite dimensional invariant subspace property for F-algebras was studied in \cite{L-W 88} with the additional assumption of ``inversely equicontinuity'' on the actions. The finite dimensional invariant subspace property considered in our main theorem (Theorem~\ref{KF thm}) in this paper removes this redundant assumption. It is the natural analogue of Ky Fan's finite dimensional subspace property (KF). We also answer an open problem posted by the first author in 1993 during a conference on Nonlinear and Convex Analysis held at Keio University, Tokyo, (see \cite[Problem 5]{Lau_Tokyo}) regarding quantum group algebras $L^1(\Gb)$.

\section{preliminaries}\label{sec 2}

For a locally convex space $E$, the dual space of $E$ is denoted by $E^*$. The action of $f\in E^*$ at $x\in E$ is denoted either by $f(x)$ or by $\la x,f\ra$. Let $A$ be a Banach algebra and let $X$ be a Banach left, right or two-sided $A$-module. Then $X^*$ is respectively a Banach right, left or two-sided $A$-module with the corresponding module action(s) defined naturally by 
\[
\la x, f\cdot a\ra = \la a\cdot x, f\ra, \la x, a\cdot f\ra = \la x\cdot a, f\ra \quad (a\in A, f\in X^*, x\in X).
\]
Let $X$ be a Banach $A$-bimodule. A linear mapping $D$: $A\to X$ is called a \emph{derivation} if it satisfies
\[
 D(ab) = a\cdot D(b) + D(a)\cdot b \quad (a,b\in X).
 \]
 Derivations in the form $D(a) = a\cdot x_0 - x_0\cdot a$ ($a\in A$) for some fixed $x_0\in X$ are called inner definitions.
 %The projective tensor product of two Banach spaces $X$ and $Y$ is denoted by $X\hot Y$. Given a Banach algebra $A$, the product mapping $\pi$: $A\hot A\to A$ is the continuous linear mapping induced by $\pi(a\otimes b) = ab$ ($a,b\in A$). The mapping $\pi$ is in fact a Banach $A$-bimodule morphism when $A\hot A$ is equipped with the natural $A$-bimodule actions defined by $a\cdot b\otimes c = ab\otimes c$ and $b\otimes c\cdot a = b\otimes ca$ ($a,b,c\in A$).

A Banach algebra $A$ is an \emph{F-algebra} \cite{Lau_F} (also known as Lau algebras \cite{Pier}) if it is the (unique) predual of a $W^*$-algebra $ \Mf$ and the identity $\e$ of $ \Mf$ is a multiplicative linear functional on $A$.  Since $A^{**} =  \Mf^*$, we denote by $P_1(A^{**})$, the set of all normalized positive linear functionals on $ \Mf$, that is 
\[
P_1(A^{**}) =\{m\in A^{**}: m\geq 0, m(\e) =1\}.
\]
In this case $P_1(A^{**})$ is a semigroup with the first (or second) Arens multiplication).

 Examples of F-algebras include the predual algebras of a Hopf von Neumann algebra (in particular, quantum group algebras), the group algebra $L^1(G)$ of a locally compact group $G$, the Fourier algebra $A(G)$ and the Fourier-Stieltjes algebra $B(G)$ of a topological group $G$ (see \cite{D-L-S, Lau_F, Lau-Ludwig}). They also include the measure algebra $M(S)$ of a locally compact semigroup $S$. Moreover, the hypergroup algebra $L^1(H)$ and the measure algebra $M(H)$ of a locally compact hypergroup $H$ with a left Haar measure are F-algebras. In this case, it was shown in \cite[Theorem~5.2.2]{Willson} (see also \cite[Remark~5.3]{Willson2}) that $(L^1(H))^*=L^\infty(H)$ is not a Hopf von Neumann algebra unless $H$ is a locally compact group.% F-algebras are often referred to as Lau algebras (see \cite{Pier}).

We recall that a \emph{semitopological semigroup} is a semigroup $S$ with a Hausdorff topology such that, for each $a\in S$, the mappings $s\mapsto as$ and $s\mapsto sa$ from $S$ to $S$ are continuous. $S$ is called a \emph{topological semigroup} if the mapping $(s,t)\mapsto st$: $S\times S \to S$ is continuous when $S\times S$ is equipped with the product topology.

Let $G$ be a locally compact group with a fixed left Haar measure $\lambda$. Then the group algebra $L^1(G)$ is the Banach space of $\lambda$-integrable functions with product
\[
f*g(x) = \int_G{f(y)g(y^{-1}x)}d\lambda(y) \quad (x\in G).
\]
If $S$ is a locally compact semigroup, then the measure algebra $M(S)$ is the space of regular Borel measures on $S$ with the total variation norm and the convolution product (see \cite{HR}) defined by 
\[
\la\mu*\nu, f\ra = \iint_{G\times G}{f(xy)}d\nu(y)d\mu(x) \quad (f\in C_0(S)),
\]
where $C_0(S)$ denotes the space of continuous functions vanishing at $\infty$.
If $S$ is discrete, then $M(S)=\ell^1(S)$.

In his seminal paper \cite{Eymard} P. Eymard has associated to a locally compact group $G$ two important commutative Banach algebras. These are the Fourier algebra $A(G)$ and the Fourier-Stieltjes algebra $B(G)$. The latter is indeed the linear span of the set of all continuous positive definite complex-valued functions on $G$. This is also the space of the coefficient functions of the unitary representations of the group $G$. More precisely, given $u\in B(G)$ there exists a unitary representation $\pi$ of $G$ and two vectors $\xi$ and $\eta$ in the representation Hilbert space $H(\pi)$ of $\pi$ such that 
\[
u(x) = \la\pi(x)\xi, \eta\ra \quad (x\in G).
\]
Equipped with the norm $\|u\| = \inf_{\xi,\eta}\|\xi\|\,\|\eta\|$ and the pointwise multiplication $B(G)$ is a commutative Banach algebra, where the infimum is taking on all $\xi$ and $\eta$ satisfying the preceding equality. As a Banach algebra $B(G)$ is also the dual space of the group C*-algebra $C^*(G)$. The Fourier algebra $A(G)$ is the closed ideal of $B(G)$ generated by the elements of $B(G)$ with compact supports. The algebra $A(G)$ can also be defined as the set of coordinate functions of the left regular representations of $G$ on $L^2(G)$. When $G$ is abelian, via Fourier transform we have
\[
A(G) = L^1(\hat G), \quad B(G) = M(\hat G) \quad \text{and } C^*(G) = C_0(\hat G),
\]
where $\hat G$ is the dual group of $G$.

Let $S$ be a semitopological semigroup. Let $C_b(S)$ be the commutative C*-algebra of all bounded continuous complex-valued functions on $S$ with the sup norm topology and the pointwise product. Consider its C*-subalgebra $LUC(S)$ consisting of all left uniformly continuous functions on $S$ (i.e. all
$f\in C_b(S)$ such that the mapping $s\mapsto \ell_sf$ from $S$ into $C_b(S)$ is continuous when $C_b(S)$ has the norm topology, where $(\ell_sf)(t) = f(st)$ for $s,t\in S$). Evidently, $LUC(S)$ 
is invariant under translations and contains the constant function. The semigroup $S$ is called \emph{left amenable} if $LUC(S)$ has a left invariant mean, that is there is $m\in LUC(S)^*$ such that 
\[
\|m\| = m(1) =1,\; m(\ell_sf) = m(f) \quad (s\in S, f\in LUC(S)).
\]
% where $\ell_s$ is the left translation operator on $LUC(S)$ by $s$. 
We call $S$  \emph{extremely left amenable} if there is a left invariant mean $m$ which is multiplicative, that is it satisfies further
 \[
 m(fg) = m(f)m(g)\quad (f,g \in LUC(S)).
 \]

Let $A$ be an F-algebra. Elements of $P_1(A^{**})$ are called \emph{means} on $A^* =\Mf$. It is well-known that 
\[
P_1(A) = P_1(A^{**})\cap A
\]
  is a topological semigroup with the product and topology carried from $A$ (this may be regarded as a consequence of \cite[Proposition~1.5.2]{Sakai}), and $P_1(A)$ spans $A$ (see \cite[Theorem~1.14.3]{Sakai}). A mean $m \in P_1(A^{**})$ on $A^*$ is called a \emph{topological left invariant mean}, abbreviated as TLIM,  if $a\cdot m = m$ for all $a\in P_1(A)$, in other words $m \in P_1(A^{**})$ is a TLIM if $m(x\cdot a) = m(x)$ for all $a\in P_1(A)$ and $x\in A^*$. 

 An F-algebra $A$ is called \emph{left amenable} if, for each Banach $A$-bimodule $X$ with the left module action specified by $a\cdot x = \langle a, \e\rangle x$ ($a\in A$, $x\in X$), every continuous derivation from $A$ into $X^*$ is inner. The following was shown in \cite{Lau_F} (see Theorems~4.1 and 4.6 there).

\begin{lemma}\label{left amen1-3}
Let $A$ be an F-algebra. Then the following are equivalent.
\begin{enumerate}
\item There is a TLIM for $A^*$. \label{TLIM}
\item  %If $X$ is a Banach $A$-bimodule with the left module action given by $a\cdot x = \langle a, \e\rangle x$, then every derivation $D$: $A\to X^*$ is inner
The algebra $A$ is left amenable.\label{inner}
\item There exists a net $(m_\al)\subset P_1(A)$ such that $am_\al - m_\al \to 0$
in norm topology for each $a\in P_1(A)$.\label{appTLIM}
\end{enumerate}
\end{lemma}
 
We note here that, being F-algebras, the group algebra $L^1(G)$, the measure algebra $M(G)$ of a locally compact group $G$ are left amenable if and only if $G$ is an amenable group; while the Fourier algebra $A(G)$ and the Fourier-Stieltjes algebra $B(G)$ are always left amenable \cite{Lau_F}. The hypergroup algebra $L^1(H)$ of a locally compact hypergroup $H$ with a left Haar measure is left amenable if and only if $H$ is an amenable hypergroup \cite{Skan}. Also the left amenability of the predual algebra of a Hopf von Neumann algebra, as an F-algebra, coincides with that studied in \cite{Ruan, Voic} (see also \cite{B-T} and references therein).

\begin{remark}
Suppose that $A$ and $B$ are F-algebras associated to W*-algebras $\Mf_1$ and $\Mf_2$ respectively. Let $\e_i$ be the identity of $\Mf_i$ ($i=1,2$). It is readily seen that if there is a Banach algebra homomorphism $P$: $A\to B$ such that $P^*(\e_2) = \e_1$ and $P(A)$ is dense in $B$, then the left amenability of $A$ implies the left amenability of $B$.
\end{remark}

\section{Fixed point property}\label{sec 3}

Let $S$ be a semigroup and let $C$ be a Hausdorff space. We say that $\Sc = \{T_s:\; s\in S\}$ is a \emph{representation} of $S$ on $C$ if for each $s\in S$, $T_s$ is a mapping from $C$ into $C$ and $T_{st}(x) = T_s(T_tx)$ ($s,t\in S$, $x\in C$). Sometimes we simply use $sx$ to denote $T_s(x)$ if there is no ambiguity in the context. Suppose that $S$ is a semitopological semigroup. We say that the representation is separately (resp. jointly) continuous if the mapping $(s,x) \mapsto T_s(x)$ from $S\times C$ into $C$ is separately (resp. jointly) continuous. If $C$ is a locally convex space $E$ with topology generated by a family $Q$ of seminorms, we denote it by $C=(E,Q)$.

A \emph{metric semigroup} is a semitopological semigroup whose topology is a metric $d$. We consider the following fixed point property for a metric semigroup $S$.

\begin{description}
\item[($F_U$)] If $ \Sc = \{T_s: s\in S\}$ is a separately continuous representation of $S$ on a compact subset $K$ of a locally convex space $(E,Q)$ and if the mapping $s\mapsto T_s(y)$ from $S$ into $K$ is uniformly continuous for each $y\in K$, then $K$ has a common fixed point for $S$.\label{FU}
\end{description}

Note that the mapping $s \mapsto T_s(y)$ is \emph{uniformly continuous} if for each $\tau\in Q$ and each $\epsilon >0$ there is $\delta>0$ such that 
\[
\tau(T_s(y) - T_t(y)) \leq \epsilon
\]
 whenever $d(s,t) \leq \delta$. For example, suppose that $S$ is a subset of a locally convex space $L$ that acts on $E$ such that $(a,y)\mapsto ay$: $L\times E \to E$ is separately continuous; if $a\mapsto ay$ is linear in $a\in L$ for each $y\in E$, then the induced action of $S$ on $E$, $(s,y)\mapsto sy$ ($s\in S$), is uniformly continuous in $s$ for each $y\in E$.

Let $A$ be an F-algebra. As we have known, $P_1(A)$ is indeed a metric topological semigroup with the product and topology inherited from $A$.

\begin{lemma}\label{fixed pt}
The F-algebra $A$ is left amenable if and only if the metric semigroup $P_1(A)$ has the fixed point property ($F_U$).
\end{lemma}
\proof
If $A$ is left amenable, from Lemma~\ref{left amen1-3}.(\ref{appTLIM}) there exists a net $(m_\al)\subset P_1(A)$ such that 
\[
am_\al - m_\al \to 0
\]
 in norm topology for each $a\in P_1(A)$. Fix a $y\in K$. We may assume, without loss of generality, $\lim_\al T_{m_\al} (y) = y_0 \in K$ due to the compactness of $K$. Then for $a\in P_1(A)$, by separate continuity we have 
 \[
 T_a(y_0) = \lim_\al T_a(T_{m_\al}(y)).
 \]
  If, in addition, the representation satisfies that $s\mapsto T_s(y)$ is uniformly continuous, then
\[
T_a(y_0) - y_0 = \lim_\al T_{am_\al}(y) - T_{m_\al}(y) = 0 \quad (a\in P_1(A)).
\]
Thus $y_0$ is a common fixed point for $P_1(A)$. This shows that($F_U$) holds for $P_1(A)$.

Conversely, suppose ($F_U$) holds for  $P_1(A)$. Let $(E ,\text{wk*})$ be  $A^{**}$ with the weak* topology, and let $K=P_1(A^{**})$, the set of all means on $A^*$. $K$ is a compact subset of $E$. The canonical representation of $P_1(A)$ on $K$ induced by the left $A$-module action on $A^{**}$ is clearly separately continuous. The mapping $s\mapsto sm$ is norm-norm uniformly continuous and hence is also norm-weak* uniformly continuous. Therefore there is a common fixed point $m_0\in K$ for $P_1(A)$. This $m_0$ is indeed a TLIM on $A^*$. So $A$ is left amenable.

\qed

%The reader should compare Lemma~\ref{fixed pt} with the equivalences $(1) \Leftrightarrow (3) \Leftrightarrow(4)$ of \cite[Theorem~7.7]{Shawn}, which was aroused by \cite[Theorem~4.2]{Wong}, and \cite{N-N}.

Granirer showed in \cite{Gran_ELA} that a discrete semigroup is extremely left amenable if and only if any two elements of it have a common right zero (see \cite[Theorem~4.2]{Lau-Z} for a short proof). It is due to  Mitchell \cite{Mitch_LUC} that a semitopological semigroup $S$ is extremely left amenable if and only if it has the following fixed point property.
\begin{description}
\item[($F_E$)] Every jointly continuous representation of $S$ on a compact Hausdorff space $C$ has a common fixed point in $C$.
\end{description}

 For an F-algebra $A$ it is pleasing that the left amenability of $A$ is equivalent to the extreme left amenability of $P_1(A)$ as revealed by the following theorem. 

\begin{thm}\label{ELA}
Let $A$ be an F-algebra. Then $A$ is left amenable if and only if $P_1(A)$ has the fixed point property ($F_E$).
\end{thm}
\proof
We denote $P_1(A)$ by $S$. Let $\Delta$ be the spectrum of $LUC(S)$, that is
\[ 
\Delta = \{\mu\in LUC(S)^*: \|\mu\| = \mu(1) =1, \mu(fg) = \mu(f)\mu(g) \text{ for } f,g \in LUC(S)  \}.
\]
It is evident that $\Delta$ is a compact subset of the locally convex space $(LUC(S)^*, \text{wk*})$. For $s\in S$ define 
\[
T_s(\mu) = \ell_s^*(\mu)\quad (s\in S, \mu \in \Delta),
\]
 where $\ell_s^*$ is the dual operator of the left translate $\ell_s$ by $s$. Then $ \Sc=\{T_s: s\in S\}$ is a representation of $S$ on $\Delta$. Since $\ell_s^*$ is weak* continuous, the representation is separately continuous. Furthermore, if $s_1, s_2 \in S$, $f\in LUC(S)$ and $\mu \in \Delta$, we have
\[
\la T_{s_1}(\mu) - T_{s_2}(\mu), f\ra = \la \mu, \ell_{s_1}f - \ell_{s_2}f \ra.
\]
Since $f\in LUC(S)$, for any $\epsilon>0$ there is $\delta >0$ such that 
\[
\|\ell_{s_1}f - \ell_{s_2}f\| < \epsilon
\]
 whenever $\|s_1-s_2\|_1 <\delta$. This shows that the representation of $S$ on $\Delta$ is uniformly continuous in $s$. If $A$ is left amenable, then  there is $\mu \in \Delta$ such that $\ell_s^*(\mu) = \mu$ for all $s\in S$ due to Lemma~\ref{fixed pt}. Clearly, this $\mu$ is a multiplicative left invariant mean on $LUC(S)$. So $S$ is extremely left amenable.

Conversely, if $S$ is extremely left amenable then ($F_E$) holds from  Mitchell's Theorem \cite{Mitch_LUC}. We consider $K = P_1(A^{**})$. With the weak* topology of $A^{**}$, $K$ is a compact Hausdorff space.
Consider the representation of $S$ on $K$ induced by the left $A$-module action on $A^{**}$. We show that the representation is jointly continuous, that is 
\[
(s,m)\mapsto s\cdot m:\; S\times K \to K
\]
 is continuous.
In fact, for each $f\in A^*$ and $s,s_0\in S$,
\begin{align*}
|\la f, s\cdot m - s_0 \cdot m_0 \ra| &= |\left\la f, (s-s_0)\cdot m+ s_0 \cdot (m-m_0)\right \ra| \\
                                      &\leq  \|f\| \,\|s-s_0\| + |\la f\cdot s_0, m - m_0 \ra|
\end{align*}
If $s\to s_0$ in the norm topology of $A$ and $m\to m_0$ in the weak* topology of $A^{**}$, the right side tends to 0. So the representation is jointly continuous. From ($F_E$) there is a common fixed point in $K$ for $S$ which gives a TLIM on $A^*$.

\qed

As we is known (see  the paragraph after Lemma~\ref{left amen1-3} or \cite{Lau_F}), a locally compact group $G$ is amenable if and only $L^1(G)$ is left amenable as an F-algebra. For a semigroup $S$, $S$ is a left amenable semigroup if and only $\ell^1(S)$ is a left amenable F-algebra. According to Theorem~\ref{ELA}, we therefore can characterize amenability of a group (resp.  semigroup) in terms of the fixed point property of normalized positive functions in the group/semigroup algebra.

\begin{cor}\label{amen gp semigp}
Let $G$ be a locally compact group and let $S$ a semigroup. Then
\begin{enumerate}
\item The group $G$ is amenable if and only if the metric semigroup $P_1(G)=\{f\in L^1(G), f\geq 0, \|f\|_1 = 1\}$ has the fixed point property ($F_E$).
\item The semigroup $S$ is left amenable if and only if the metric semigroup $P_1(S)=\{f\in \ell^1(S), f\geq 0, \|f\|_1 = 1\}$ has the fixed point property ($F_E$).
\end{enumerate}
\end{cor}

\begin{remark}
For $A=A(G)$, the Fourier algebra of a locally compact group $G$, the necessity part of Theorem~\ref{ELA} was obtained in \cite{Lau_AG}.
\end{remark}

\begin{remark}
From Theorem~\ref{ELA}, a locally compact group is amenable if and only if $LUC(S_G)$ has a multiplicative left invariant mean, where $S_G$ is the metric semigroup $P_1(G) =\{f\in L^1(G): f\geq 0, \la f, 1 \ra =1\}$. As a consequence, $AP(S_G)$ has a left invariant mean (which is equivalent to the left reversibility of $\overline{(S_G)^a}$, the almost periodic compactification of $S_G$) if  $G$ is amenable. %Is this something we wanted to show by using the Emerson's Theorem but did not succeed?
\end{remark}

\begin{remark}
A common fixed point property for affine actions of $P_1(A)$ with a weak topology on compact convex sets has been studied in \cite{D-N-N} for left amenable F-algebras $A$.
\end{remark}

\section{Finite dimensional invariant subspaces}\label{sec 4}

Let $E$ be a separated locally convex vector space and $X$ a subset of $E$. Given an integer $n>0$ we denote  by $\Lc_n(X)$ the collection of all $n$-dimensional subspaces of $E$ that are included in $X$. Let $S$ be a semigroup and $\Sc =\{T_s: s\in S\}$ a linear representation of $S$ on $E$. We say that $X$ is \emph{$n$-consistent} with respect to $S$ if $\Lc_n(X)\neq \emptyset$ and $\Lc_n(X)$ is $S$-invariant, that is $T_s(L) \in \Lc_n(X)$ for all $s\in S$ whenever $L\in \Lc_n(X)$. We say that the representation $\Sc$ is \emph{jointly continuous on compact sets} if the following is true: For each compact set $K\subset E$, if $(s_\al) \subset S$ and $(x_\al)\subset K$ are such that $s_\al \overset{\al}{\to}s\in S$,  $x_\al \overset{\al}{\to}x\in K$ and $T_{s_\al}(x_\al)\in K$ for all $\al$, then $T_{s_\al}(x_\al)\overset{\al}{\to}T_s(x)$. Obviously, if the mapping $(s,x) \mapsto T_s(x)$: $S\times E \to E$ is continuous, then $\Sc$ is jointly continuous on compact sets.

We are now ready to establish our main result.

\begin{thm}\label{KF thm}
Let $A$ be an F-algebra. If $A$ is left amenable then $S=P_1(A)$ has the following $n$-dimensional invariant subspace property for each $n>0$.
\begin{description}
\item[($F_n$)] Let $E$ be a separated locally convex vector space and $\Sc$ a linear representation of $S=P_1(A)$ on $E$ such that the mapping $s\mapsto T_s(x)$ is continuous for each fixed $x\in E$ and $\Sc$ is jointly continuous on compact subsets of $E$. If $X$ is a subset of $E$ $n$-consistent with respect to $S$, and if there is a closed $S$-invariant subspace $H$ of $E$ with codimension $n$ such that $(x+H)\cap X$ is compact for each $x\in E$, then there is $L_0\in \Lc_n(X)$ such that $T_s(L_0) = L_0$ ($s\in S$). \end{description}
Conversely, if ($F_1$) holds then $A$ is left amenable. Moreover, for any $n>0$, the property ($F_{n+1}$) implies the property ($F_n$).
\end{thm}
\proof

Let $Y = E/H$ and let $q$: $E\to Y$ be the quotient map. Then $Y$ is an $n$-dimensional locally convex space. Since $H\cap X$ is compact, It must be true that $q|_L$ is injective and hence $q(L) = Y$ for each $L\in \Lc_n(X)$. Fix a basis 
\[
\vec{y}=(y_1,y_2,\cdots,y_n)
\]
 of $Y$. Then each $L\in \Lc_n(X)$ has a unique basis 
 \[
 \vec x_L=(x^{(L)}_1,x^{(L)}_2,\cdots,x^{(L)}_n)
 \]
   such that $q(\vec x_L) = \vec y$. Denote 
   \[
   X_i= q^{-1}(y_i)\cap X \quad (i=1,2,\cdots,n).
   \]
    From the hypothesis each $X_i$ is a compact subset of $E$. Let 
\[
\Kc = \{\vec x_L: L\in \Lc_n(X)\}.
\]
 Then 
 \[
 \Kc\subset \Pi_{i=1}^n X_i
 \]
  and is closed. In fact, if $\vec x_\al\in \Kc$ and $\vec x_\al \stackrel{\al}{\to}\vec x$ then $\vec x\in \Pi_{i=1}^n X_i$ since the latter is compact due to Tychynoff Theorem. Any linear combination $\sum_{i=1}^n{c_ix_i}$ of the components of $\vec x$ is the limit of the net of  same linear combinations of $(\vec x_\al)$ which is included in $q^{-1}(\sum_{i=1}^n{c_iy_i})\cap X$. As the intersection of a fixed coset of $H$ and $X$ the last set is compact according to the assumption. So the linear span of $\vec x$ belongs to $X$. Since $q(\vec x)=\vec y$, we have $\vec x\in \Kc$. Thus, $\Kc$ is a compact Hausdorff space. We now define an $S$-action $\theta_s$ ($s\in S$) on $\Kc$ by 
\[\theta_s(\vec x_L) = \vec x_{T_s(L)} \quad (s\in S, L\in \Lc_n(X)).
\]
 It is evident that $\Lambda = \{\theta_s: s\in S\}$ is a representation of $S$ on $\Kc$. We show that it is also jointly continuous.

We denote the natural $S$-action on $Y$ inherited from that on $E$ by $\tilde T_s$ ($s\in S$). The action of $\tilde T_s$ is in fact formulated by
\[
\tilde T_s( y) = q(T_s( x)) \quad \text{if } y=q(x)
\]
for each $y\in Y$. This well defines $\tilde T_s(y)$ since if $q(x) = q(x') = y$ then $x-x'\in H$ and then $T_sx-T_sx' = T_s(x-x')\in H$, which ensures that $q(T_s(x)) = q(T_s(x'))$. Apply $\tilde T_s$ to $\vec y$. we get
\begin{equation}\label{tilde T}
\tilde T_s(\vec y) = q(T_s(\vec x)) \quad (\vec x \in \Kc),
\end{equation}
that is $\tilde T_s(q(\vec x)) = q(T_s(\vec x))$  ($\vec x \in \Kc$).

Since $q(T_s(L)) = Y$ and $T_s(\vec x)$ is a basis of $T_s(L)$ if $\vec x\in \Kc$ is a basis of $L\in \Lc_n(X)$, we have that $\tilde T_s(\vec y)$ is a basis of $Y$. Therefore $\tilde T_s$ is an invertible operator on $Y$. We have the following relation
\begin{equation}\label{theta}
\theta_s(\vec x) = T_s\circ q_L^{-1}\circ \tilde T_s^{-1} (\vec y) \quad (\vec x\in \Kc),
\end{equation}
where $q_L^{-1}$ is the inverse of $q|_L$: $L \stackrel{\text{onto}}{\longrightarrow} Y$ and $L$ is the unique element of $\Lc_n(X)$ containing $\vec x$ as its basis. In fact, the left side of (\ref{theta}) is, by definition, $\vec x_{T_s(L)}$, the only basis of $T_s(L)$ satisfying $q(\vec x_{T_s(L)}) = \vec y$; while the right side of (\ref{theta}) is also a basis of $T_s(L)$ and
\[
q\left(T_s\circ q_L^{-1}\circ \tilde T_s^{-1} (\vec y)\right) = \tilde T_s\circ q\left( q_L^{-1}\circ \tilde T_s^{-1} (\vec y)\right)= \vec y
\]
by identity (\ref{tilde T}).
This shows that the right side of (\ref{theta}) is indeed $\vec x_{T_s(L)}$. So the identity (\ref{theta}) holds.
The mapping $s\mapsto \tilde T_s$ from $S$ to $B(Y)$, the algebra of bounded operators on $Y$, is also continuous. To clarify this, since $Y$ is finite dimensional, it suffices to show $s\mapsto \la \tilde T_s(y),f\ra$ is continuous for each $y\in Y$ and $f\in Y^*$. Take an $x\in E$ such that $q(x) = y$. We have
\[
 \la \tilde T_s(y),f\ra = \la T_s(x),f\circ q\ra.
\]
Obviously, the right side is continuous in $s$ since $s\mapsto T_s(x)$ is continuous (only weak continuity is needed) and $f\circ q\in E^*$. Therefore, the mapping $s\mapsto \tilde T_s$ is continuous. As a consequence, the mapping $s\mapsto \tilde T_s^{-1}$ is continuous. So the matrix representation $M(s)$ of $\tilde T_s^{-1}$ associated to the basis $\vec y$ is continuous in $s$. We have
\[
\theta_s(\vec x) = T_s\circ q_L^{-1} \left(M(s) \vec y^T\right) =T_s( M(s)\vec x^T)  \quad (\vec x\in \Kc).
\]
Let $(\vec x_\al)\subset \Kc$ and $(s_\al)\subset S$ be such that $\vec x_\al \overset{\al}{\to} \vec x\in \Kc$ and $s_\al \overset{\al}{\to}s \in S$. Then %$T_{s_\al}(\vec x_\al)) \overset{\al}{\to} T_s(\vec x)$ 
\[
\theta_{s_\al}(\vec x_\al) = T_{s_\al}( M(s_\al)\vec x_\al^T) \overset{\al}{\to} T_s( M(s)\vec x^T) = \theta_s(\vec x)
\]
since the original representation $\Sc$ is jointly continuous on compact sets (Here the compact set is taken to be $\Kc\cup \overline{M((s_\al))}\Kc$ and we note $\theta_{s_\al}(\vec x_\al)\in \Kc$ for all $\al$). This shows that $\Lambda$ is a jointly continuous representation on $\Kc$. By Theorem~\ref{ELA}, we conclude that there is a common fixed point, say $\vec x_0$, for $\Lambda$ in $\Kc$. Let $\vec x_0$ be a basis of $L_0\in \Lc_n(X)$. From the definition of $\theta_s$ the above implies that the basis $\vec x_0$ of $L_0$ is also a basis of $T_s(L_0)$ ($s\in S$). This leads to $L_0 = T_s(L_0)$ for all $s\in S$. Thus ($F_n$) holds.

For the converse, we assume ($F_1$) holds. Consider the left $A$-module $E= A^{**}=(A^*)^*$ with the weak* topology, and consider the linear representation of $S=P_1(A)$ on $E$ defined by the left $A$-module morphism: 
\[
T_s(u) = s\cdot u\quad (s\in S).
\]
 It is easy to check that the representation is jointly continuous on compact sets of $(E, \text{wk*})$.  Let 
 \[
 X = \cup_{u\in P_1(A^{**})} \Cbb u.
 \]
  Then $\Lc_1(X) = \{ \Cbb u: u\in P_1(A^{**})\}$ is $1$-consistent with respect to $S$. Now define 
  \[
  H = \{u\in E: \la u, \e\ra = 0\}.
  \]
   This is a $1$-codimensional $S$-invariant closed subspace of $E$. We have 
   \[
   (u+H)\cap X = \la u, \e \ra P_1(A^{**}),
   \]
    which is compact for each $u\in E$. By ($F_1$) there is $m\in P_1(A^{**})$ such that $T_s( \Cbb m) =  \Cbb m$ ($s\in S$). So $s\cdot m = cm$ for some $c\in  \Cbb$. Since $m$ and $s\cdot m$ belong to $P_1(A^{**})$, we must have $c=1$. Thus $s\cdot m = m$ for all $s\in S= P_1(A)$. Therefore, $m$ is a TLIM on $A^*$. So $A$ is left amenable.

We now show $(F_{n+1}) \Rightarrow (F_n)$. Suppose that $E$, $X$, $H$ are as described in ($F_n$). We aim to show that there is $L_0\in \Lc_n(X)$ such that \[
T_s(L_0) = L_0\quad (s\in S),
\]
 assuming that $(F_{n+1})$ holds. We consider $\tilde E := E\oplus \Cbb$ with the product topology, and extend the $S$-action on $E$ to $\tilde E$ by defining 
 \[
 T_s(x+c) = T_s(x) + c\quad (x\in E, c\in\Cbb).
 \]
  The extended $S$-action is still jointly continuous on compact sets. Let 
\[
\tilde X = X \oplus \Cbb : = \{(x,c): x\in X, c\in \Cbb\}.
\]
 It is readily seen that $\tilde X$ is (n+1)-consistent with respect to $S$. Now $H$ is a closed $(n+1)$-codimensional $S$-invariant subspace of $\tilde E$ and 
\[
[(x+c)+H]\cap \tilde X = (x+H)\cap X + c 
\]
is compact for each $x\in E$ and $c\in \Cbb$. From ($F_{n+1}$), there is $\tilde L_0\in \Lc_{n+1}(\tilde X)$ such that $T_s(\tilde L_0) = \tilde L_0$ ($s\in S$). Since $\Lc_n(X)$ does not contain any (n+1)-dimensional subspace of $E$ (otherwise $H\cap X$ will contain a nontrivial subspace of $E$ which is contradict to the assumption that this intersection is compact), $\tilde L_0 = L_0 +\Cbb$ with $L_0$ an $n$-dimensional subspace contained in $X$. Certainly $L_0\in \Lc_n(X)$ and $T_s(L_0) = L_0$ ($s\in S$).

\qed

\begin{remark}
When S is a semitopological semigroup, Ky Fan's property (KF) was denoted by $P(n)$ in \cite{Lau_invar}. Our proof of Theorem 4.1 for the part $(F_{n+1})\Rightarrow (F_n)$ still works to show $P(n+1)\Rightarrow P(n)$. It implies that if S has property $P(n)$ then it has property $P(1)$ and so S is left amenable. This, in turn, implies that $P(n+1)$ (in fact $P(m)$ for all integers $m$) holds. Therefore, we have answered affirmatively the question posed in \cite[Page 376]{Lau_invar} (for the group case, see \cite[Theorem 4.2]{L-P-W}).
\end{remark}

We now consider some special cases.

\begin{exa}
For a discrete semigroup $S$, any left invariant mean on $\ell^\infty(S)$ is necessarily a topological left invariant mean on it (and vice versa). So Theorem~\ref{KF thm} implies the following:
If $S$ is left amenable then, for all $n>0$, the convolution semigroup algebra $\ell^1(S)$ has the finite dimensional invariant subspace property ($F_n$). 
\end{exa}

\begin{exa}\label{group}
For a locally compact group $G$, It follows from Greenleaf \cite[Theorem~2.2.1]{Greenleaf} and Wong \cite[Theorem 3.3]{Wong 71} that amenability of $G$ is equivalent to the left amenability of $L^1(G)$ and $M(G)$. So the group algebra $L^1(G)$ and the measure algebra $M(G)$ satisfy ($F_n$) for each $n>0$ if $G$ is amenable. %Conversely, if either $L^1(G)$ or $M(G)$ satisfies ($F_n$) for some $n>0$, then $G$ is amenable.
\end{exa}

\begin{exa}\label{Fourier}
From \cite{Eymard} and \cite{Lau-Ludwig}, the Fourier algebra $A(G)$ of a locally compact group $G$ and the Fourier Stieltjes algebra $B(G)$ of a topological group $G$ are commutative F-algebras. Denote both by $A$. Since $P_1(A)$ is commutative, $\ell^\infty(P_1(A))$, the space of all bounded complex-valued functions on the discrete semigroup $P_1(A)$, has an invariant mean. By \cite[Lemma~2.1]{L-W 88}, $A$ is left amenable. therefore, $A(G)$ and $B(G)$ always have the invariant subspace property ($F_n$) for all $n>0$.
\end{exa}

We recall that a Hopf von Neumann algebra is a pair $(\Mf, \Gamma)$, where $\Mf$ is a von Neumann algebra and $\Gamma$: $\Mf \to \Mf\bar\otimes \Mf$ is a co-multiplication, that is a normal, unital *-homomorphism satisfying
\[ (id\otimes \Gamma)\circ \Gamma =  (\Gamma\otimes id)\circ \Gamma.
\]
Here $\Mf\bar\otimes \Mf$ denotes the von Neumann algebra tensor product of $\Mf$ with itself. The pre-adjoint $\Gamma_*$: $\Mf_*\hot \Mf_* \to \Mf_*$  of the co-multiplication $\Gamma$ induces a product, denoted by $*$, on the unique predual $\Mf_*$ of $\Mf$:
\[
\la a*b, f\ra = \la a\otimes b, \Gamma(f)\ra \quad (a,b\in \Mf_*, f\in \Mf).
\]
With this product $\Mf_*$ becomes an F-algebra. Left amenability of $\Mf_*$ has been studied in \cite{Ruan, Voic}. Applying Theorem~\ref{KF thm}, we derive the following characterization result.

\begin{cor}\label{Hopf}
The predual algebra $\Mf_*$ of the Hopf von Neumann algebra $(\Mf, \Gamma)$ is left amenable if and only if the topological semigroup $P_1(\Mf_*)$ has the fixed point property ($F_E$). Furthermore, $\Mf_*$ is left amenable if and only if the semigroup $P_1(\Mf_*)$ has the finite invariant subspace property ($F_n$) for some integer (equivalently, all integers) $n\geq 1$.
\end{cor}

We recall further that a locally compact quantum group is a Hopf von Neumann algebra $(\Mf, \Gamma)$ such that there exist a normal semifinite faithful left invariant weight $\phi$ and a  normal semifinite faithful right invariant weight $\psi$ on $(\Mf, \Gamma)$. We denote it by $\Gb = (\Mf, \Gamma, \phi, \psi)$ (see \cite{K-V}). In this case the F-algebra $\Mf_*$ is called the quantum group algebra of $\Gb$, denoted by $L^1(\Gb)$; while the von Neumann algebra $\Mf$ in the case is usually written as $L^\infty(\Gb)$. When $G$ is a locally compact group, Consider $\Gamma_G$: $L^\infty(G)\to L^\infty(G\times G) = L^\infty(G)\bar\otimes L^\infty(G)$ defined by 
\[
\Gamma_G(f)(x,y) = f(xy) \quad (f\in L^\infty(G), x,y\in G).
\]
Let $\phi$ be a left Haar measure and let $\psi$ be a right Haar measure on $G$. Then $\Gb = (\Mf, \Gamma_G, \phi, \psi)$ is a locally compact quantum group and $L^1(\Gb) = L^1(G)$. Likewise, $VN(G)$ is a locally compact quantum group and in this case $L^1(\Gb) = A(G)$.

 From Corollary~\ref{Hopf} we immediately derive the following:

\begin{cor}\label{quantum}
Let $\Gb$ be a locally compact quantum group. Then $L^1(\Gb)$ is left amenable if and only if the topological semigroup $P_1(L^1(\Gb))$ has the fixed point property ($F_E$) if and only if $P_1(L^1(\Gb))$ has the finite invariant subspace property ($F_n$) for some integer (and then for all integers) $n\geq 1$.
\end{cor}

 Corollary~\ref{quantum} settles the Open Problem 5 of \cite{Lau_finite}.
 
 We note that Examples~\ref{group} and \ref{Fourier} are special cases of Corollary~\ref{quantum} since $L^1(G)$ and $A(G)$ are quantum group algebras. However, the direct arguments we indicated there without detouring through quantum groups are much more natural and elementary. We also note that if $m\in VN(G)^*$ is a topological left invariant mean, then the 1-dimensional subspace $\langle m\rangle$ spanned by $m$ is closed left ideal in the Banach algebra $VN(G)^*=A(G)^{**}$ with the Arens product (see \cite{F-M, F-M-N} for the study of ideals in the bidual of $A(G)$).

A locally compact space $H$ is a hypergroup if there is a ``convolution product'', denoted by $*$, defined on $M(H)$, the space of bounded Radon measures on $H$, with which several general conditions are satisfied. We refer to \cite{Bloom-Heyer} for the precise definition of a locally compact hypergroup.  With the convolution product $M(H)$ becomes a Banach algebra. When a locally compact hypergroup $H$ has a left invariant Haar measure $\lambda$, the convolution product on $L^1(H)=L^1(H,\lambda)$ is then naturally defined to make it a Banach algebra, called the hypergroup algebra of $H$. One may consider amenability of $H$ in terms of the existence of a left invariant mean on $L^\infty(H)$ \cite{Skan}. 
In general, $L^\infty(H)= L^1(H)^*$ cannot be a Hopf von Neumann algebra (in particular, it cannot be a quantum group) that induces $L^1(H)$ as its predual algebra unless $H$ is a locally compact group \cite{Willson} (also see \cite{Willson2}). However, it is a von Neumann algebra and induces $L^1(H)$ to an F-algebra. Left amenability of the F-algebra $L^1(H)$ is the same as (left) amenability of the hypergroup $H$ \cite{Skan}. So we have the following.

\begin{cor}\label{hyper}
Let $H$ be a locally compact hypergroup. If $H$ is amenable, then the hypergroup algebra $L^1(H)$ and the measure algebra $M(H)$ satisfy the finite dimensional invariant subspace property ($F_n$) for each $n>0$. Conversely, if either $L^1(H)$ or $M(H)$ satisfies ($F_n$) for some $n>0$, then $H$ is amenable.
\end{cor}

%\proof
%The same as that of Corollary~\ref{group}.

%\qed

\section{module inverse and left amenability of ideals}\label{sec 5}

There are other useful characterizations for left amenability of F-algebras. We discuss some of them in this section.

Suppose that $X$ and $Y$ are two Banach left $A$-modules. Then the Banach space $B(X,Y)$ of all bounded linear operators from $X$ into $Y$ is a Banach $A$-bimodule. The module actions are given by 
\[
(a\cdot f)(x) = a f(x) \text{ and } (f\cdot a)(x) = f(ax) \quad (f\in B(X,Y), a\in A, x\in X).
\]

\begin{lemma}\label{net}
Let $ A$ be a left amenable F-algebra.  Let $X$ and $Y$ be left Banach $ A$-modules with the left $ A$-module action on $Y$ being given by $a\cdot y =  \left\langle a,\e \right\rangle y$ ($y\in Y$). Suppose that $T\in B(X,Y)$ is a left $ A$-module morphism that has a continuous right inverse. Then there is a bounded net $(F_i)\subset B(Y,X)$ of right inverses of $T$ such that 
\[
\lim_i\|a\cdot F_i - F_i\cdot a\| = 0 \quad (a\in A).
\]
\end{lemma}
\proof
From Lemma~\ref{left amen1-3}(\ref{appTLIM}), there is a net $(m_i)\subset P_1(A)$ such that $\lim_i\|am_i - \la a,\e\ra m_i \|=0$ in the norm topology of $A$ for each $a\in A$. 
Let $F\in B(Y,X)$ be a right inverse of $T$.
   Define $F_i= m_i\cdot F$.
    Then the net $(F_i)$ is bounded in $B(Y,X)$, and each $F_i$ is a right inverse of $T$ since $T$ is a left $A$-module morphism and
\[
T\circ F_i(y) = m_i\cdot T\circ F(y) =\la m_i, \e\ra y = y \quad (y\in Y).
\]
 Moreover, 
 \[(a\cdot F_i)(y) =(am_i)\cdot F(y) \quad
  \text{and } 
  (F_i\cdot a)(y) = \la a,\e \ra  m_i \cdot F(y).
  \]
   Thus 
\[
(a\cdot F_i -F_i \cdot a)(y) = (am_i - \la a,\e \ra m_i)\cdot F(y) \quad (y\in Y)
\]
and 
\[
\|a\cdot F_i - F_i \cdot a\| \leq \|F\| \|am_i -  \la a,\e \ra m_i\| \overset{i}{\to} 0 \quad (a\in  A).
\]
\qed

\begin{remark}
 We refer to \cite[Theorem~3.1]{Zhang_survey} for an analogue of Lemma~\ref{net} concerning approximate Banach algebras. 
\end{remark}

Let $A$ be a Banach algebra. Given a multiplicative linear functional $\varphi$ on $A$, we recall that $A$ is \emph{$\varphi$-amenable} if there is $m\in A^{**}$ such that $m(\varphi) = 1$ and $a\cdot m =\varphi(a) m$ for all $a\in A$ (see \cite{%K-L-P1, K-L-P2,
H-S-T, Mehdi}). Indeed, if $A$  is an F-algebra then to say that it is left amenable is equivalent to say that it is $\e$-amenable. The concept of $\varphi$-amenability helps us to describe left amenability from the point of view of ideals of an F-algebra. We note that an ideal of an F-algebra $A$ may no longer be an F-algebra since its dual space may no longer be a W*-algebra as a quotient of $A^*$.
 
\begin{example}
Let $G$ be a locally compact group and $E \subset G$ be a proper subset of spectral synthesis. Let $A(G)$ be the Fourier algebra of $G$, and let $VN(G)$ be the group von Neumann algebra of $G$. Then $A(G)^* = VN(G)$. Let 
\[ I(E) = \{\phi\in A(G): \phi(x)=0 \text{ for all }x\in E\} \]
Then $I(E)^\perp =$ weak* closure of linear span of $\{\rho(x): x\in E\}$, where 
\[
\rho(x):\; L^2(G) \to L^2(G),\; \rho(x)h(y) = h(x^{-1}y)
\]
 for $h \in L^2(G)$ (see \cite[Lemma 7.3]{Lau-Losert}). In particular, $I(E)$ is a closed ideal of $A(G)$, and $I(E)^\perp$ is in general not a closed ideal in $VN(G)$. So $I(E)^* = VN(G)/I(E)^\perp$ (as a quotient of $VN(G)$) is not a C*-algebra. 
\end{example}

 However, restricting on the ideal, the identity $\e$ of $A^*$ is still  multiplicative. So we can consider $\e$-amenability for ideals of an F-algebra. We denote the restriction $\e|_J$ by $\e_J$.

\begin{thm}\label{left ideal}
Let $ A$ be an F-algebra.
\begin{enumerate}
\item If $A$ is left amenable and if $J$ is a closed left ideal of $ A$ such that $\e_J \neq 0$, then $J$ is $\e_J$-amenable.\label{ideal}
\item If there is a closed left ideal $J$ of $ A$ such that $\e_J \neq 0$ and $J$ is $\e_J$-amenable, then $ A$ is left amenable.\label{ideal converse}
\end{enumerate}
\end{thm}

\proof
To show (\ref{ideal}) we let $m\in  A^{**}$ be a TLIM on $ A^*$ so that  $m(\e)= 1$, $a\cdot m = \la a,\e\ra  m$ for $a\in  A$. Let $Y =  \Cbb m$. Then $Y$ is a Banach left $ A$-module with the module action determined by $a\cdot m = \la a,\e\ra  m$. Note that, since $J$ is a closed left ideal of $ A$, $J$ is naturally a Banach left $ A$-module. Consider $P$: $J\to Y$ defined by $P(j) = \la j,\e\ra m$. Then $P$ is a Banach left $ A$-module morphism. Take a $j_0\in J$ such that $\la j_0,\e\ra = 1$ and define $Q$: $Y\to J$ by $Q(cm) = cj_0$. Then $Q$ is a right inverse of $P$. In fact, 
\[
P\circ Q(cm) = cP(j_0) = c\la j_0,\e\ra m = cm\quad (c\in  \Cbb).
\]
 We now can apply Lemma~\ref{net} to obtain a bounded net $(Q_\al)\subset B(Y,J)$ of right inverses of $P$ that satisfies $a\cdot Q_\al - Q_\al\cdot a \overset{\al}{\to} 0$ for $a\in  A$. We let $j_\al = Q_\al(m)$. Then 
 \[
 \la j_\al, \e \ra m = P(j_\al) = P\circ Q_\al(m) = m.
 \]
  This shows that $\la j_\al, \e_J \ra =1$ for each $\al$. Moreover, 
  \[
  jj_\al - \la j, \e \ra j_\al = (j\cdot Q_\al - Q_\al\cdot j )(m) \overset{\al}{\to} 0\quad (j\in J).\]
   Thus $J$ is $\e_J$-amenable.% due to \cite[Theorem 1.4]{K-L-P1}.

To prove (\ref{ideal converse}) we let $\imath$: $J\to  A$ be the canonical embedding of $J$ into $ A$. Then $\imath$ is a Banach left $ A$-module morphism. So is $\imath^{**}$: $J^{**} \to  A^{**}$, the second dual operator of $\imath$. Moreover, $\imath^*(\e) = \e_J$. If $J$ is $\e_J$-amenable, then there is $m_0\in J^{**}$ such that $ m_0(\e_J) = 1$ and
\[
j\cdot m_0 = \e_J(j) m_0 = \la j, \e \ra m_0\quad (j\in J).
\]
 Take a $j_0\in J$ such that $\la j_0, \e \ra = 1$. We have 
$j_0\cdot m_0 = m_0$ and 
\[
a\cdot m_0 = (aj_0)\cdot m_0 = \la aj_0, \e_J \ra m_0 = \la a, \e \ra \la j_0, \e \ra m_0 = \la a,\e\ra m_0
\]
for $a\in  A$. Let $n = \imath^{**}(m_0)$. Then 
\[
\la n, \e\ra = \la m_0, \imath^*(\e)\ra = \la m_0, \e_J\ra = 1, 
\]
\[
a\cdot n = a \cdot \imath^{**}(m_0) = \imath^{**}(a\cdot m_0) = \la a, \e \ra\imath^{**}(m_0) = \la a, \e \ra n \quad (a\in  A).
\]
From these, using a standard method (see, e.g., the last paragraph in the proof of \cite[Theorem~4.1]{Lau_F}), one can construct a TLIM $m$ on $A^*$ from $n$. Therefore, $ A$ is left amenable.

\qed

\begin{exa}
For a locally compact group $G$, $L^1(G)$ is a closed ideal of $M(G)$ and the restriction to $L^1(G)$ of the identity of $M(G)^*$ is the constant function $1$ which is the identity of $L^\infty(G)$. From Theorem~\ref{left ideal} $L^1(G)$ is left amenable if and only if $M(G)$ is left amenable. Since $L^1(G)$ is left amenable if and only if $G$ is amenable %($m$ is a TLIM on $L^\infty(G)$ if and only if $m_0\odot m$ is a left invariant mean on it, where $m_0$ is a weak* cluster point of a positive approximate identity for $L^1(G)$ and $\odot$ denotes the first Arens product on $L^1(G)^{**}$),
(see \cite{Lau_F}), we conclude that $M(G)$ is left amenable if and only if $G$ is amenable. Note that if $G$ is not discrete, $M(G)$ is never amenable as a Banach algebra.
\end{exa}

\begin{exa}\label{A1+A2}
Let $A_1$ and $A_2$ be F-algebras and $\e_i$ the identity of the corresponding von Neumann algebra $A_i^*$ ($i =1,2$). Then $A = A_1 \oplus_1 A_2$ is an F-algebra with pointwise addition and scalar multiplication and with the product defined by
\[
(x_1,x_2)\cdot (y_1,y_2) = (x_1y_1 +\la x_2,\e_2\ra y_1 +\la y_2,\e_2\ra x_1, x_2y_2) 
\]
for $x_1,y_1\in A_1$ and $x_2,y_2\in A_2$. We have $A^* = A_1^* \oplus_\infty A_2^*$ whose identity is $\e=(\e_1,\e_2)$. It is readily seen that $A_1$ is a closed left ideal of $A$ and $\e|_{A_1} = \e_1$. From Theorem~\ref{left ideal}, the F-algebra $A$ is left amenable if and only if $A_1$ is left amenable. This result was originally obtained in \cite[Proposition~4.5]{Lau_F}. In particular, the F-algebra $A_1$ is left amenable if and only if its unitization $A_1^\sharp$ is left amenable.
\end{exa}

\begin{cor}\label{Ss}
Let $S$ be a semigroups and $s_0$ any element in $S$. Then $S$ is left amenable if and only if $S_0=Ss_0$ is left amenable.
\end{cor}
\proof 
This is simply because $\ell^1(S_0)$ is a closed left ideal of $\ell^1(S)$ and left amenability of $S$ (resp. $S_0$) is the same as left amenability of $\ell^1(S)$ (resp. $\ell^1(S_0)$).

\qed

% We recall that a subset $S'$ of a semigroup $S$ is called \emph{left thick} in $S$ if for each finite subset $F\subset S$ there exists an $s_F \in S$ such that $Fs_F\subset S'$. 

We remark that Corollary~\ref{Ss} also follows from \cite[Theorem~9]{Mitch 65}.

We now turn to some interesting consequences of Lemma~\ref{net}.

\begin{cor}\label{inverse}
The F-algebra $A$ is left amenable if and only if the following holds. 

For any Banach right $A$-modules $X$ and $Y$ with the right $A$-module action on $Y$ being given by $y\cdot a = \langle a,\e \rangle y$ ($y\in Y$),  if $T\in B(X^*,Y^*)$ is a weak*-weak* continuous left $A$-module morphism and $T$ has a continuous right inverse, then there is a continuous right inverse of $T$ which is a left $A$-module morphism.
\end{cor}
\proof
If $A$ is left amenable then, from Lemma~\ref{net}, there is a bounded net $(F_\al)\subset B(Y^*,X^*)$ such that 
\[
\|a\cdot F_\al - F_\al\cdot a\| \to 0
\]
 for $a\in A$
 and $T\circ F_\al =id_{Y^*}$. Since $B(Y^*,X^*) = (Y^*\hot X)^*$ there is a subnet of $F_\al$, still denoted by $F_\al$, such that 
 \[
 F_\al \overset{\al}{\to}  \Theta \in B(Y^*, X^*)
 \]
  in the weak* topology of $(Y^*\hot X)^*$. This implies that 
  \[
  F_\al(y^*) \overset{\text{wk*}}{\to}  \Theta(y^*)
  \]
   in $X^*$ for $y^*\in Y^*$. By the weak* continuity of $T$ we derive 
\[
T\circ  \Theta(y^*) = T(\text{wk*-}\lim_\al F_\al(y^*)) = \text{wk*-}\lim_\al T\circ F_\al(y^*) = y^*
\]
for each $y^*\in Y^*$. So $ \Theta$ is a right inverse of $T$. Moreover 
\[a\cdot  \Theta(y^*) = a\cdot \text{wk*-}\lim_\al F_\al (y^*) %= \text{wk*-}\lim_\al a\cdot F_\al (y^*)
 = \text{wk*-}\lim_\al F_\al (ay^*) =  \Theta(ay^*)
\] 
for each $y^*\in Y^*$, $a\in  A$. Thus $ \Theta$ is a left $ A$-module morphism.

Conversely, consider the Banach right $A$-module $X=A^*$ and consider $Y = \Cbb$ with the right $A$-module action $c\cdot a =c \la a, \e\ra$ ($c\in \Cbb$). Let $T$: $X^*\to Y^*$ (note $Y^* = \Cbb$) be the continuous linear operator defined by 
\[
T(u) = \la \e, u \ra\quad (u\in X^* = A^{**}).
\]
 $T$ is obviously a weak* continuous left $A$-module morphism and has a continuous right inverse. (Take a $u_0\in A^{**}$ such that 
$ \la \e, u_0\ra = 1$.
Then $F$: $c\mapsto cu_0$ is such an inverse.) From the hypothesis, there is a continuous left $A$-module morphism $\Theta$: $Y^*\to X^*$ which is a right inverse of $T$. Let $n =\Theta(1)$. Then we have 
  \[
  \la \e, n\ra = T(n) =T\circ \Theta (1) =1
  \]
   and 
\[
a\cdot n = a\cdot \Theta(1) = \Theta (a\cdot 1) = \la a, \e \ra \Theta(1) = \la a, \e\ra n \quad (a\in A).
\]
This implies that there exists a TLIM for $A^*$.
Therefore, $A$ is left amenable due to Lemma~\ref{left amen1-3}.

%Conversely, consider $X=A^*$ and $Y = \Cbb$ with the right $A$-module action on $Y$ being given by $c\cdot a =c \la a, \e\ra$ ($c\in \Cbb$). Let $T$: $X^*\to Y^*$ (note $Y^* = \Cbb$) be the continuous linear operator defined by $T(u) = \la \e, u \ra$ ($u\in X^* = A^{**}$). $T$ is obviously weak*-weak* continuous and has a continuous right inverse (take a $u_0\in A^{**}$ such that $\la \e, u_0\ra = 1$, then $F$: $c\to cu_0$ is such a inverse). From the hypothesis, there is a left $A$-modue morphism $\Theta$: $Y^*\to X^*$ which is a right inverse of $T$. Let $n =\Theta(1)$. Then we have $\la \e, n\ra = T(n) =1$ and 
%\[
%a\cdot n = \Theta (a\cdot 1) = \la a, \e \ra \Theta(1) = \la a, \e\ra n \quad (a\in A).
%\]
%From this $n$, one may construct a TLIM on $A^*$.

\qed 

%Let $\|F_\al\| \leq M$ for all $\al$. Let $K = \{h\in X^* : \|h\|\leq M\}$. Equipted with the weak* topology of $X^*$ $K$ is compact. We have $\{F_\al|_{Y^*_1}\}\subset K^{Y^*_1}$, where $Y^*_1$ is the unit ball of $Y^*$. By the compactness of $K^{Y^*_1}$, there is a subnet of $\{F_\al\}$, still denoted by $\{F_\al\}$, such that $\{F_\al|_{Y^*_1}\}$ converges to some $\Theta\in \subset K^{Y^*_1}$, i.e. $F_\al(y) \overset{\text{wk*}}{\to} \Theta(y)$ for $y\in Y^*_1$. Since $F_\al$ is linear, $\Theta$ can be extended to a linear map from $Y^*$ to $X^*$ and $F_\al(y) \overset{\text{wk*}}{\to} \Theta(y)$ for $y\in Y^*$. By the weak* continuity of $T$ we derive $T\circ \Theta(y) = T(\text{wk*-}\lim_\al F_\al(y)) = \text{wk*-}\lim_\al T\circ F_\al(y) = y$ ($y\in Y$). So $\Theta$ is a right inverse of $T$. Clearly, $\|\Theta\|\leq M$ and $a\cdot \Theta(y) = a\cdot \text{wk*-}\lim_\al F_\al (y) = \text{wk*-}\lim_\al a\cdot F_\al (y) = \text{wk*-}\lim_\al F_\al (ay) = \Theta(ay)$ ($y\in Y^*$, $a\in  A$). Thus $\Theta$ is a left $ A$-module morphism.

\begin{cor}\label{functional ext}
Let $ A$ be a left amenable F-algebra. Suppose that $X$ is a right Banach $ A$-module and $Y$ is a closed submodule of $X$. If $f\in Y^*$ is such that $a\cdot f = \langle a,\e \rangle f$ for $a\in  A$. Then $f$ extends to some $F\in X^*$ such that $a\cdot F = \langle a,\e \rangle F$ ($a\in  A$).
\end{cor}
\proof
If $a\cdot f = \langle a,\e \rangle f$ for $a\in  A$, then $ \Cbb f$ is a submodule of $Y^*$ and $Y^* =  \Cbb f \oplus Y_f^*$, where 
\[
Y_f = \{y\in Y: f(y) = 0\}
\]
 is a right $ A$-module. The decomposition shows that there is a bounded left module projection $P$: $Y^* \to  \Cbb f$ and $ \Cbb f\cong (Y/Y_f)^*$ is a dual Banach left $ A$-module. Let $T$: $X^* \to  \Cbb f$ be defined by $T(u) = P(u|_Y)$. Then $T$ is a left $ A$-module morphism. Let $f'$ be an extension of $f$ to $X$ according to the Hahn-Banach Theorem. Then $\tau(cf) = cf'$ defines a right inverse of $T$. The conclusion of the corollary follows from Corollary~\ref{inverse}.
 
\qed

\section{Some related results and open problems}\label{sec 6}

Given a Banach algebra $A$ we denote the spectrum of $A$ by $\Delta(A)$, i.e. $\Delta(A)$ is the set of all multiplicative linear functionals on $A$. All results in the previous section can be easily extended to $\varphi$-amenability cases (see the paragraph after Lemma~\ref{net} for the definition of $\varphi$-amenability for a Banach algebra). Here we only highlight some of them below. Related investigation for $\varphi$-invariant functionals may be seen in \cite{Filali92} (see also \cite{B-F, Filali99} for applications in studying finite-dimensional ideals in some algebras associated to a locally compact group). 

Similar to Theorem~\ref{left ideal} we can characterize $\varphi$-amenability of $A$ in terms of that of ideals.

\begin{prop}\label{phi ideal}
Let $A$ be a Banach algebra and $\varphi\in \Delta(A)$.
\begin{enumerate}
\item If $A$ is $\varphi$-amenable and if $J$ is a closed left ideal of $ A$ such that $\varphi_J \neq 0$, then $J$ is $\varphi_J$-amenable.
\item If there is a closed left ideal $J$ of $ A$ such that $\varphi_J \neq 0$ and $J$ is $\varphi_J$-amenable, then $ A$ is $\varphi$-amenable.
\end{enumerate}
\end{prop}

As a consequence of Proposition~\ref{phi ideal}, we see immediately that $A$ is $\varphi$-amenable if and only if $A^\#$, the unitization of $A$, is $\varphi_1$-amenable, where $\varphi_1$ is the unique character extension of $\varphi$ to $A^\#$. One can also derive a similar result concerning multiplier algebras. Recall that if $A$ has a bounded approximate identity, then $A$ is a closed ideal of its multiplier algebra $M(A)$. For example, the group algebra $L^1(G)$ is a closed ideal of the measure algebra $M(G)$ for a locally compact group $G$; and the algebra $\Kc(H)$ of compact operators on a Hilbert space $H$ is a closed ideal of $M(\Kc(H))=\Bc(H)$. For any $\varphi\in \Delta(A)$, there is a unique character $\tilde\varphi\in \Delta(M(A))$ that extends $\varphi$ (\cite[Proposition~1.4.27]{Dales}). We have the following.
\begin{cor}
Let $A$ be a Banach algebra with a bounded approximate identity and let $\varphi\in \Delta(A)$. Then $A$ is $\varphi$-amenable if and only if its multiplier algebra $M(A)$ is $\tilde\varphi$-amenable.
\end{cor}

If $A$ is an F-algebra associated to a von Neumann algebra $\Mf$, then, as the predual of an operator space, it is naturally an operator space. But $A$ may not be a completely contractive Banach algebra with this operator space structure. However, in many important cases $A$ is indeed completely contractive. For example, it is well known that $A$ is completely contractive if it is the predual algebra of a Hopf von Neumann algebra \cite{Ruan}. It is still possible that $A$ is  completely contractive without being a predual algebra of a Hopf von Neumann algebra. For example, as well known, in general a semigroup algebra $\ell^1(S)$ is not a predual algebra of a Hopf von Neumann algebra. But it is still completely contractive. To see this one only needs to notice that as a commutative von Neumann algebra $\ell^\infty(S) = \min  \ell^\infty(S)$. Then $\ell^1(S) = \max \ell^1(S)$ (see \cite[Section~3.3]{E-R} for detail). So the operator space projective tensor product $\ell^1(S) \hot \ell^1(S)$ is the same as the Banach space projective tensor product  $\ell^1(S) \hot^\gamma \ell^1(S)$. Thus $\ell^1(S) \hot \ell^1(S) = \ell^1(S\times S) = \max \ell^1(S\times S)$ which implies that $\ell^1(S)$ is completely contractive. 

In the sequel we use the standard notations as used in the monograph \cite{E-R}, In particular, for operator spaces $V$ and $W$,  $V\hot W$  denotes the operator space projective tensor product of $V$ and $W$. If $V$ and $W$ are preduals of von Neumann algebras $\Rf$ and $\Sf$ respectively, then $V\hot W$ turns out to be the predual of $\Rf \overline\otimes \Sf$, the spacial von Neumann algebra tensor product of $\Rf$ and $\Sf$ \cite[Theorem~7.2.4]{E-R}. 

Now for the F-algebra $A$, denote by $\pi$: $A\hot A \to A$ the multiplication mapping. Its dual mapping is $\pi^*$: $\Mf \to \Mf \overline\otimes \Mf$. We have the following fact.

\begin{lemma}
$\pi$ is completely bounded (completely contractive) if and only if $\pi^*$ is so.
\end{lemma}

\begin{proof}
For each integer $n$, consider the induced $\pi_n$: $M_n(A\hot A) \to M_n(A)$. We have
\[
\la\la \pi_n(\al(a\otimes b)\beta), f \ra\ra = \la\la \al(a\otimes b)\beta, (\pi^*)_n f \ra\ra \quad (f\in M_n(\Mf))
\]
for $a\in M_p(A)$, $b\in M_q(A)$ and scalar matrices $\al\in M_{n,p\times q}$ and $\beta\in M_{p\times q, n}$. This duality formula leads one directly to the claimed equivalence.
 
\end{proof}

We note that $\pi^*$ is a unital (i.e. $\pi^*(1)= 1\otimes 1$) and co-associative (i.e. $(id\otimes \pi^*)\circ \pi^* = (\pi^*\otimes id)\circ \pi^*$) mapping. But $\pi^*$ is usually not an algebra homomorphism.

In the case when $A$ is a completely contractive F-algebra, one then can consider operator amenability for $A$. This is a weak version of amenability compare to B. E. Johnson's Banach algebra amenability. However if the associated von Neumann algebra $\Mf$ is commutative, then operator amenability of $A$ is the same as Banach algebra amenability for $A$ (This can be regarded as a consequence of \cite[Proposition~2.5]{Ruan_A(G)}). In particular, $\ell^1(S)$ is operator amenable if and only if it is Banach algebra amenable, which turns out to be a very strong condition for a semigroup $S$. Regarding the relation between the left amenability and the operator amenability for a completely contractive F-algebra we have the following general result.

\begin{prop}
Let $A$ be a completely contractive F-algebra. If $A$ is operator amenable then it is left amenable.
\end{prop}

\begin{proof}
If $A$ is operator amenable, by \cite[Proposition 2.4]{Ruan_A(G)} there is $M\in (A\hot A)^{**} = (\Mf\overline\otimes \Mf)^*$ such that $a\cdot M = M\cdot a$ and $\pi^{**}(M)\cdot a = a$ for all $a\in A$). Define $n\in A^{**} =\Mf^*$ by $n(x) = \la M, x\otimes 1\ra$ ($x\in \Mf$). Then we have
\[
n(1) =  \la M, 1\otimes 1\ra =  \la M, \pi*(1)\ra = \pi^{**}(M)(1).
\]
Now take $a\in A$ such that $\la a, 1\ra = 1$. We obtain
\begin{align*}
 \pi^{**}(M)(1) &=  \pi^{**}(M)(1)\la a, 1\ra =\la \pi^{**}(M), a\cdot 1)\ra \\
                &= \la \pi^{**}(M)\cdot a, 1\ra =\la a, 1\ra = 1.
\end{align*}
This shows $n(1) = 1$. On the other hand, for all $a\in A$ and $x\in \Mf$
\begin{align*}
n(a\cdot x) &= \la M, a\cdot x\otimes 1 \ra = \la M\cdot a, x\otimes 1 \ra = \la a\cdot M, x\otimes 1\ra \\
&=  \la M,  x\otimes 1\cdot a \ra = \la a, 1 \ra  \la M,  x\otimes 1 \ra = \la a, 1 \ra n(x)
\end{align*}
From this $n$, with a standard construction one can get a left invariant mean on $\Mf$.

\end{proof}

We note that the above result was obtained for the Hopf von Neumann algebra case by Z.-J. Ruan in \cite[Theorem~2.1]{Ruan}, where left amenability was called Voiculescu amenability.
 
We conclude this paper with several open questions as follows.

\begin{problem}
Let $A$ be an F-algebra. Let ($F'_n$) denote the same property as ($F_n$) with ``jointly continuous'' replaced by ``separately continuous'' on compact subsets of $E$. Does ($F'_n$) imply ($F_n$)?
\end{problem}

Let $A$ be an F-algebra. Regard $A$ as the Banach $A$-bimodule with the module multiplications being given by the product of $A$. Then the dual space $A^*$ is a Banach $A$-bimodule. %Then $f\in A^*$ and $a\in A$ define the elements $f\cdot a$ and $a\cdot f$ in $A^*$ by 
%\[
%\la a\cdot f, x\ra = \la f, xa\ra, \quad \text{and }\la f\cdot a, x\ra = \la f,  a\cdot x\ra
%\]for all $a,x\in A$. 
We say that a subspace $X$ of $A^*$ is topologically left (resp. right) invariant if $a\cdot X \subset X$ (resp. $X\cdot a\subset X$) for each $a\in A$; We call $X$ topologically invariant if it is both left and right topological invariant. An element of $A^*$ is almost periodic (resp. weakly almost periodic) if the map $a\mapsto f\cdot a$ from $A$ into $A^*$ is a compact (resp. weakly compact) operator. Let $AP(A)$ and $WAP(A)$ denote the collection of almost periodic and weakly almost periodic functions on $A$ respectively. Then $AP(A)$ and $WAP(A)$ are closed topologically invariant subspaces of $A^*$. Furthermore, $1\in AP(A) \subset WAP(A)$. When $G$ is a locally compact group and $A= L^1(G)$, then $AP(A)= AP(G)$ and $WAP(A)= WAP(G)$, where $AP(G)$ and  $WAP(G)$ are spaces of, respectively,  almost periodic and  weakly almost periodic continuous functions on $G$ (see \cite{Lau 87} and \cite{Lau-Z, L-Z1} for more details concerning these spaces). 

Let ($F_n^A$) denote the same property as ($F_n$) with joint continuity replaced by  equicontinuity on compact subsets of $E$. It is known that if $A$ satisfies ($F_n^A$) then $AP(A)$ has a TLIM that is an element $m\in AP(A)^*$ such that $\|m\|=\la m,\e\ra =1$ and $\la m, f\cdot a\ra = \la m, f\ra$ for all $a\in A$ and $f\in AP(A)$ (see \cite{Lau_AG}). 

\begin{problem}
Does the existence of TLIM on $AP(A)$ imply ($F_n^A$) for all $n\geq 1$?
\end{problem}

Let ($F_n^W$) denote the same property as ($F_n^A$) with equicontinuity on compact subsets of $E$ replaced by quasi-equicontinuity on compact subsets of $E$ (which means the closure of $S$ in the product space $E^K$, for each compact set $K\subset E$, consists only of continuous maps from $K$ to $E$). We have known that if $A$ satisfies ($F_n^W$) for each $n\geq 1$ then $WAP(A)$ has a TLIM (see \cite{Lau_AG}).   

\begin{problem}
Does the existence of TLIM on $WAP(A)$ imply ($F_n^W$) for all $n\geq 1$?
\end{problem}

\begin{acknowledgment}
The authors would like to thank the referee for his/her carefull reading of the paper and valuable suggestions.
\end{acknowledgment}

\bibliographystyle{amsplain}

\begin{thebibliography}{99}

\bibitem{B-F}
J. W. Baker and M. Filali, On minimal ideals in some Banach algebras associated with a locally compact group, J. London Math. Soc. (2) 63 (2001), 83-98.


\bibitem{B-T}
E. B\'edos and  L. Tuset, Amenability and co-amenability for locally compact quantum groups, Internat. J. Math. 14 (2003), 865-884. 

\bibitem{Bloom-Heyer}
W. R. Bloom and H. Heyer,
Harmonic analysis of probability measures on hypergroups, de Gruyter Studies in Mathematics 20, Walter de Gruyter \& Co., Berlin, 1995.

\bibitem{E-R}
E. G. Effros and Z.-J. Ruan, Operator spaces, Lond. Math. Soc. Monographs new series 23, Clarendon, Oxford, 2000.

\bibitem{Dales}
H. G. Dales, Banach algebras and automatic continuity, Clarendon Press, Oxford, 2000.

\bibitem{D-L-S}
H. G. Dales, A. T.-M. Lau and D. Strauss,
\emph{Second duals of measure algebras}, Dissertations Math 481, 2012.  

%\bibitem{Shawn}
%S. Desaulniers,  Geometry and fixed point properties for a class of Banach algebras associated to locally compact groups, PhD thesis, University of Alberta,  2008

\bibitem{D-N-N}
S. Desaulniers, R. Nasr-Isfahani and M. Nemati, Common fixed point properties and amenability of a class of Banach algebras, preprint.

\bibitem{Eymard}
P. Eymard,  Sur les applications qui laissent stable l'ensemble des fonctions presque-periodiques, Bull. Soc. Math. France 89 (1961) 207-222.

\bibitem{K3}
K. Fan, Invariant subspaces for a semigroup of linear operators, Nederl. Akad. Wetensch. Proc. Ser. A 68, Indag. Math. 27 (1965), 447-451.

\bibitem{K1}
K. Fan, Invariant cross-sections and invariant linear subspaces, Israel J. Math. 2 (1964), 19-26.

\bibitem{K2}
K. Fan, Invariant subspaces of certain linear operators, Bull. Amer. Math. Soc. 69 (1963), 773-777.

\bibitem{Filali99}
M. Filali, Finite-dimensional left ideals in some algebras associated with a locally compact group, Proc. Amer. Math. Soc.  127  (1999), 2325-2333.

\bibitem{Filali92}
M. Filali, The ideal structure of some Banach algebras, Math. Proc. Cambridge Philos. Soc.  111 (1992), 567-576.

\bibitem{F-M-N}
M. Filali, M. Neufang and M. Sangani Monfared, On ideals in the bidual of the Fourier algebra and related algebras, J. Funct. Anal. 258 (2010), 3117-3133.

\bibitem{F-M}
M. Filali and M. Sangani Monfared,
 Finite-dimensional left ideals in the duals of introverted spaces,
 Proc. Amer. Math. Soc. 139 (2011), 3645-3656
 
\bibitem{Gran_ELA}
E. Granirer, Extremely amenable semigroups, Math. Scand. {17} (1965), 177-197.

\bibitem{Greenleaf}
F. P. Greenleaf, 
 Invariant means on topological groups and their applications. {\it Van Nostrand Mathematical Studies, No. 16} Van Nostrand Reinhold Co., New York-Toronto, Ont.-London, 1969.

\bibitem{HR}
E. Hewitt and K. A. Ross, Abstract Harmonic Analysis I, Springer-Verlag, New York,  1963.

\bibitem{H-S-T}
Z. Hu, M. Sangani Monfared and T. Traynor, On character amenable Banach algebras, Studia Math. 193 (2009), 53-78.

\bibitem{Iohv}
I. S. Iohvidov,
Unitary operators in a space with an indefinite metric, Zap. Mat. Otd. Fiz.-Mat. Fak. i Har'kov. Mat. Obsc. (4) 21 (1949), 79-86.

\bibitem{I-K}
I. S. Iohvidov and M. G. Krein, Spectral theory of operators in spaces with indefinite metric I, Amer. Math. Soc. Transl. (2) 13 (1960) 105-175.

%\bibitem{K-L-P1}

%E. Kaniuth, A. T.-M. Lau, J. Pym, On $\varphi$-amenability of Banach algebras, Math. Proc. Camb. Phil. Soc. 144(2008), 85-96.

%\bibitem{K-L-P2}
%E. Kaniuth, A. T.-M. Lau, J. Pym, On character amenability of Banach algebras, J. Math. Anal. Appl. 344 (2008), 942-955.

\bibitem{Krein}
M. G. Krein, On an application of the fixed-point principle in the theory of linear transformations of spaces with an indefinite metric, Amer. Math. Soc. Transl. (2) 1 (1955), 27-35. 

\bibitem{K-V}
J. Kustermans and S. Vaes, Locally compact quantum groups, Ann. Sci. Ecole Norm. Sup. (4) 33 (2000), 837-934. 


\bibitem{Lau_finite}
A. T.-M. Lau, Finite dimensional invariant subspace properties and amenability, J. Nonlinear Convex Anal. 11 (2010), 587-595.


\bibitem{Lau_Tokyo}
A. T.-M. Lau, Fixed point and finite-dimensional invariant subspace properties for semigroups and amenability. Nonlinear and convex analysis in economic theory (Tokyo, 1993), 203-213, 
Lecture Notes in Econom. and Math. Systems, 419, Springer, Berlin, 1995. 

\bibitem{Lau_AG}
A. T.-M. Lau, Fourier and Fourier-Stieltjes algebras of a locally compact group and amenability, Topological vector spaces, algebras and related areas (Hamilton, ON, 1994), 79-92, Pitman Res. Notes Math. Ser. 316, Longman Sci. Tech., Harlow, 1994.

\bibitem{Lau 87}
A. T.-M. Lau, Uniformly continuous functionals on Banach algebras, Colloq. Math. 51 (1987), 195-205.

\bibitem{Lau_F}
A. T.-M. Lau, Analysis on a class of Banach algebras with applications to harmonic analysis on locally compact groups and semigroups, Fund. Math. 118 (1983), 161-175.

\bibitem{Lau_invar}
A. T.-M. Lau, Finite-dimensional invariant subspaces for a semigroup of linear operators, J. Math. Anal. Appl. 97 (1983), 374-379.

\bibitem{Lau-Losert}
A. T.-M. Lau and V. Losert, The C*-algebra generated by operators with compact support on a locally compact group, J. Funct. Anal. 112 (1993), 1-30.

\bibitem{Lau-Ludwig}
A. T.-M. Lau and J. Ludwig, Fourier-Stieltjes algebra of a topological group, Adv. Math. 229 (2012), 2000-2023. 


\bibitem{L-P-W}
A. T.-M. Lau, A. L. T. Paterson and J. C. S. Wong, Invariant subspace theorems for amenable groups, Proc. Edinburgh Math. Soc. 32 (1989), 415-430.

\bibitem{L-W 88}
A. T.-M. Lau and J. C. S. Wong, Invariant subspaces for algebras of linear operators and amenable locally compact groups, Proc. Amer. Math. Soc. 102 (1988), 581-586.

\bibitem{Lau-Wong}
A. T.-M. Lau and J. C. S. Wong, Finite dimensional invariant subspaces for measurable semigroups of linear operators, J. Math. Anal. Appl. 127 (1987), 548-558.

\bibitem{Lau-Z}
A. T.-M. Lau and Y. Zhang, Fixed point properties for semigroups of nonlinear mappings and amenability, J. Funct. Anal. 263 (2012), 2949-2677.

\bibitem{L-Z1}
A. T.-M. Lau and Y. Zhang, 
Fixed point properties of semigroups of non-expansive mappings. J. Funct. Anal.  254  (2008), 2534-2554.

\bibitem{Mehdi}
M. Sangani Monfared, Character amenability of Banach algebras, Math. Proc. Cambridge Philos. Soc. 144 (2008), 697-706. 

\bibitem{Mitch_LUC}
T. Mitchell, Topological semigroups and fixed points, {\it Illinois J. Math.} 14 (1970), 630-641.

\bibitem{Mitch 65}
T. Mitchell, Constant functions and left invariant means on semigroups, Trans. Amer. Math. Soc. 119 (1965) 244-261.

\bibitem{Naimark1}
M. A. Naimark, On commuting unitary operators in spaces with indefinite metric. Acta Sci. Math. (Szeged) 24 (1963) 177-189.

\bibitem{Naimark2}
M. A. Naimark, Commutative unitary permutation operators on a $\Pi_k$ space, Soviet Math. 4 (1963) 543-545. 

\bibitem{Pier}
J.-P. Pier, Amenable locally compact groups, Pitman Research in Math. 172, Longman Group UK LImited, 1988.

\bibitem{Pontr}
L. Pontrjagin, Hermitian operators in spaces with indefinite metric,  Bull. Acad. Sci. URSS. Ser. Math. [Izvestia Akad. Nauk SSSR] 8 (1944) 243-280.

\bibitem{Ruan}
Z.-J. Ruan,
Amenability of Hopf von Neumann algebras and Kac Algebras, J. Funct. Anal. 139 (1996) 466-499.

\bibitem{Ruan_A(G)}
Z.-J. Ruan, The operator amenability of $A(G)$, Amer. J. Math. 117 (1995), 1449-1474.

\bibitem{Sakai}
S. Sakai, C*-algebras and W*-algebras, Springer Verlag, 1971.

\bibitem{Skan}
M. Skantharajah, Amenable hypergroups. Illinois J. Math. 36 (1992) 15-46.

\bibitem{Voic}
D. Voiculescu, Amenability and Katz algebras, Algebres d'operateurs et leurs applications en physique mathematique (Proc. Colloq., Marseille, 1977), 451-457, Colloq. Internat. CNRS, 274, CNRS, Paris, 1979. 

\bibitem{Willson}
B. Willson,
Invariant nets for amenable groups and hypergroups, Ph.D. thesis, University of Alberta, 2011.

\bibitem{Willson2}
B. Willson, Configurations and invariant nets for amenable hypergroups and related algebras, Trans. Amer. Math. Soc., to appear.

\bibitem{Wong 71}
J. C. S. Wong, Topological invariant means on locally compact groups and fixed points, Proc. Amer. Math. Soc. 27 (1971) 572-578.

%\bibitem{Wong}
%J. C. S. Wong, Invariant means on locally compact semigroups, Proc. Amer. Math. Soc. 31 (1972) 39-45.

\bibitem{Zhang_survey}
Y. Zhang,  Solved and unsolved problems on generalized notions of amenability for Banach algebras, Banach algebras 2009, 441-454, Banach Center Publ., 91, Polish Acad. Sci. Inst. Math., Warsaw, 2010.

\end{thebibliography}

\end{document}